\documentclass{amsart}
\usepackage{xcolor}
\usepackage{amsmath,amssymb,amsthm,amsfonts,amscd,mathrsfs,mathtools,enumerate}
\usepackage{tikz,graphicx}
\usepackage[all]{xy}
\usepackage[hyphens]{url}
\usepackage{hyperref}
\usepackage{caption}
\usetikzlibrary{decorations.markings}

\theoremstyle{plain}
    \newtheorem{thm}{Theorem}[section]
    \newtheorem{lem}[thm]   {Lemma}
    \newtheorem{cor}[thm]   {Corollary}
    \newtheorem{prop}[thm]  {Proposition}
    
\theoremstyle{definition}
    \newtheorem{defn}[thm]  {Definition}

    \newtheorem{ex}[thm]{Example}
    \newtheorem{rem}[thm]{Remark}

\def\max{\mathrm{max}}

\def\del{\partial}

\def\Ext{\mathrm{Ext}}

\def\Grp{\mathsf{Grp}}
\def\Set{\mathsf{Set}}

\def\cat{\mathsf{cat}}

\def\ker{\mathrm{ker}}

\newcommand{\be}{\begin{enumerate}}
\newcommand{\ee}{\end{enumerate}}

\newcommand{\Z}{\mathbb{Z}}

\newcommand{\F}{\mathcal{F}}
\renewcommand{\L}{\mathcal{L}}
\newcommand{\FIN}{\mathcal{FIN}}
\newcommand{\VCYC}{\mathcal{VCYC}}

\newcommand{\FG}{\mathcal{G}}
\newcommand{\FD}{\mathcal{D}}
\newcommand{\ul}{\underline}

\newcommand{\TC}{{\sf TC}}

\newcommand{\cld}{{\sf cd}}
\newcommand{\gld}{{\sf gd}}

\newcommand{\cldG}{{\sf cd}_\mathcal{G}}

\newcommand{\OD}{\mathcal{O}_\mathcal{D}}
\newcommand{\OF}{\mathcal{O}_\mathcal{F}}

\newcommand{\G}{\Gamma}
\newcommand{\GG}{G\rtimes\G}

\newcommand{\OGG}{\mathcal{O}_\FG (\GG)}

\newcommand{\OP}{\operatorname}
\newcommand{\la}{\langle}
\newcommand{\ra}{\rangle}

\makeatletter
\@namedef{subjclassname@2020}{\textup{2020} Mathematics Subject Classification}
\makeatother

\begin{document}

\title[Equivariant cohomological dimensions]{Comparison of equivariant cohomological dimensions}

\author{Mark Grant}

\author{Kevin Li}

\author{Ehud Meir}

\author{Irakli Patchkoria}

\address{Institute of Mathematics,
Fraser Noble Building,
University of Aberdeen,
Aberdeen AB24 3UE,
UK}

\address{Fakult\"{a}t f\"{u}r Mathematik, Universit\"{a}t Regensburg, 93040 Regensburg,
Germany}

\email{mark.grant@abdn.ac.uk}

\email{kevin.li@mathematik.uni-regensburg.de}

\email{ehud.meir@abdn.ac.uk}

\email{irakli.patchkoria@abdn.ac.uk}

\date{\today}


\begin{abstract} We compare three definitions of the equivariant cohomological dimension of a group with operators, coming from Takasu, Adamson and Bredon relative group cohomologies, giving examples of strict inequality in all cases where it can occur. We prove and make use of Stallings--Swan type results which characterise the groups of equivariant cohomological dimension one. Some of our examples are relevant to Farber's problem which asks for an algebraic characterisation of the topological complexity of discrete groups. In particular, the topological complexity of a group is not in general given by a relative cohomological dimension of the product relative to the diagonal subgroup. 
\end{abstract}

\maketitle

\section{Introduction}\label{sec:intro}

Let $G$ and $\G$ be discrete groups, such that $\G$ acts on $G$ by automorphisms. In the older literature $G$ may be described as a \emph{group with operators in $\G$}; here we simply say that $G$ is a \emph{$\G$-group}. The equivariant group cohomology of $\G$-groups has been defined and studied in several articles, including \cite{CG-CO, In1, GMP, In2}. There are three alternative definitions, all of which are based on some form of relative cohomology of the pair $(\GG,\G)$ consisting of the semi-direct product with $\G$ viewed as a subgroup in the standard way. There arise three corresponding definitions of the equivariant cohomological dimension of a $\G$-group. The purpose of this article is to compare these dimensions, giving examples of strict inequality wherever possible. In order to provide such examples, we employ equivariant Stallings--Swan type results which characterise those $\G$-groups having equivariant cohomological dimension one.

Let $(K,L)$ be a subgroup pair, and let $M$ be a $K$-module (by which we mean a $\Z K$-module). There are two classical definitions of relative group cohomology: the version defined by Auslander \cite{Aus} and Massey \cite{Massey} and later studied by Takasu \cite{Takasu1,Takasu2}, which we denote by $H^*(K,L;M)$ and call \emph{Takasu cohomology}; and the version defined by Adamson \cite{Adamson} and later interpreted by Hochschild \cite{Hoch} in terms of relative homological algebra, which we denote by $H^*([K:L];M)$ and call \emph{Adamson cohomology}. We refer to Section \ref{sec:defns} below for more details, and to the introduction to \cite{A-NC-M} for a historical account with further references. Subsequent to Bredon's seminal work on equivariant cohomology \cite{Bredon2}, a third version of relative group cohomology with more general coefficient modules arises as follows. Let $\mathcal{L}$ denote the family of subgroups of $K$ which are conjugate to a subgroup of $L$. One can do homological algebra in the abelian category of modules over the restricted orbit category $\mathcal{O}_\mathcal{L}(K)$.  Given an $\mathcal{O}_\mathcal{L}(K)$-module $\ul{M}$, one can then define the \emph{Bredon cohomology} $H^*_\mathcal{L}(K;\ul{M})$ in a standard way using projective resolutions of the trivial module $\ul{\Z}$. Again we refer to Section \ref{sec:defns} for more details.

Each definition of relative group cohomology gives rise to a definition of relative cohomological dimension, as the largest degree in which the relative cohomology is nonzero for some coefficient module (or $\infty$ if no such largest degree exists). These will be denoted by $\cld(K,L)$, $\cld([K:L])$ and $\cld_\L(K)$ and called respectively the \emph{Takasu dimension}, \emph{Adamson dimension} and \emph{Bredon dimension} of the pair $(K,L)$.

When $G$ is a $\G$-group as in the first paragraph, the relative group cohomology of the pair $(\GG,\G)$ is called the \emph{equivariant group cohomlogy} of $G$. The Takasu version was first studied by Cegarra, Garc\'{i}a-Calcines and Ortega in \cite{CG-CO}, while the Adamson version was first studied by Inassaridze in \cite{In1}. In deference to these authors, we call $\cld(\GG,\G)$ the \emph{Cegarra--Calcines--Ortega dimension}, and $\cld([\GG:\G])$ the \emph{Inassaridze dimension}, of the $\G$-group $G$. Letting $\mathcal{G}$ denote the family of subconjugates of $\G$ in $\GG$, we call $\cld_\mathcal{G}(\GG)$ the \emph{Bredon dimension} of the $\G$-group $G$.

It is known that Adamson cohomology can be viewed as a special case of Bredon cohomology, where we take coefficients in the $\mathcal{L}$-fixed points of a $K$-module \cite{Nuc,PY,BC-VEB,A-NC-MSS}. Therefore $\cld([K:L])\leq \cld_\L(K)$ always holds, and in the context of equivariant group cohomology we have $\cld([\GG:\G])\leq \cld_\FG(\GG)$. Inassaridze conjectured in \cite[\S 8]{In2} that strict inequality can occur. We prove this conjecture in the affirmative in Section \ref{sec:AvsB} below (see Examples \ref{ex:C6} and \ref{ex:3n4n}).

 \begin{thm}\label{thm:A<B}
There exist $\G$-groups $G$ with:
\begin{enumerate}[(a)]
\item $\cld([\GG:\G])=1$ and $\cldG(\GG)=\infty$;
\item $\cld([\GG:\G])=3n$ and $\cldG(\GG)=4n$, for any $n\ge1$.
\end{enumerate}
\end{thm}

All of these examples arise from equivariant Bestvina--Brady groups, and we make heavy use of a  theorem of Leary and Nucinkis \cite{Leary-Nucinkis} concerning their non-abelian cohomology. The examples in Theorem \ref{thm:A<B}(a) have $G$ free and $\G$ finite, and we prove and make use of the following characterisation of $\G$-groups having Inassaridze dimension less than or equal to one. This should be compared with the classical theorem of Stallings \cite{Sta} and Swan \cite{Swa} which characterises free groups as the groups having cohomological dimension less than or equal to one.

 \begin{thm}\label{thm:Inassaridze1}
 Let $G$ be a $\G$ group with $\G$ finite. Then $\cld([\GG:\G])\leq 1$ if and only if for every prime $p$ and $p$-subgroup $P$ of $\G$, the group $G$ is free with basis a $P$-set. 
 \end{thm}
 
 Theorem \ref{thm:A<B}(b) is deduced from work of Mart\'{i}nez-P\'{e}rez \cite{M-P} and Degrijse and Petrosyan \cite{DP}, who gave examples of groups whose cohomological dimension relative to the family $\FIN$ of finite subgroups is less than their Bredon dimension with respect to $\FIN$. These groups are semi-direct products associated to equivariant Bestvina--Brady groups, and we use non-abelian cohomology to show that $\FIN=\FG$ in these examples.

In Section 4 we turn to a comparison of Takasu and Bredon dimensions. We remark that in a recent pair of papers \cite{A-NC-M, A-NC-MSS} Arciniega-Nev\'{a}rez, Cisneros-Molina and S\'{a}nchez Salda\~{n}a have studied a natural comparison map between Takasu and Adamson cohomology. They showed that malnormality is a necessary and sufficient condition for this map to be an isomorphism. Since our focus here is on dimensions, there appears to be very little overlap between their work and ours. We are also interested more specifically in equivariant group cohomology, which amounts to considering pairs $(K,L)$ with the additional assumption that $L$ is a retract of $K$. While it is relatively easy to give examples where the Takasu dimension exceeds the Bredon dimension for pairs with $L$ normal in $K$, we are able to produce a large class of non-normal examples using the following Stallings--Swan type result for Takasu equivariant group cohomology, given as Theorem \ref{thm:Stallings-Swan-Takasu} below.

\begin{thm}\label{thm:Stalling-Swan-Takasu}
Let $G$ be a $\G$-group. Then $\cld(G\rtimes\G,\G)\le1$ if and only if $G$ is free with basis a free $\G$-set.
\end{thm}

Here we make use of a characterisation of Takasu dimension due to Alonso \cite{Alonso}, which in dimension one is largely due to Dicks \cite{Dicks}. This result should be compared with the equivariant Stallings--Swan theorem for Bredon equivariant group cohomology \cite[Theorem 1.5/4.4]{GMP}, which states that a $\G$-group $G$ with $\G$ finite has $\cldG(\GG)\le1$ if and only if $G$ is free with basis a (not necessarily free) $\G$-set. In fact, there are two possible notions of ``$\G$-free $\G$-group", corresponding to the images of the left adjoints of the two possible forgetful functors $U:\G$-$\Grp\to\Set$ and $U:\G$-$\Grp\to\G$-$\Set$. The first corresponds to Cegarra--Calcines--Ortega equivariant cohomological dimension one, the second to Bredon equivariant dimension one.

As a corollary to Theorem \ref{thm:Stalling-Swan-Takasu}, we obtain that the projective objects in the category of $\G$-groups are precisely the free $\G$-groups with basis a free $\G$-set, see Corollary \ref{cor:projectives}.

To conclude the paper we describe examples where the Bredon dimension exceeds the Cegarra--Calcines--Ortega dimension. This is relatively easy without the retraction assumption (i.e., for general relative group cohomology rather than equivariant group cohomology). 

\begin{prop}\label{prop:Klein}
Let $G$ be the fundamental group of the Klein bottle. Regard $G$ as a $G$-group, acting on itself by conjugation. Then $\cld(G\rtimes G,G)=4$ and $\cldG(G\rtimes G)=\infty$.
\end{prop}

This example has relevance for Farber's problem from \cite{Far06}, which asks for an algebraic characterisation of the topological complexity of discrete groups. Recall that $\TC(X)$ is defined to be the sectional category \cite{Schwarz} of the free path fibration $X^I\to X\times X$. It is a numerical homotopy invariant satisfying the bounds $\cat(X)\leq \TC(X)\leq \cat(X\times X)$ in terms of the Lusternik--Schnirelmann category \cite{Far03}. When $X$ is aspherical these invariants depend only on the fundamental group $G:=\pi_1(X)$ and we write $\TC(G)$ and $\cat(G)$ for $\TC(X)$ and $\cat(X)$. It follows from results of Eilenberg--Ganea \cite{EG}, Stallings \cite{Sta} and Swan \cite{Swa} that $\cld(G)\leq \TC(G) \leq \cld(G\times G)$. These bounds are often not sharp, and examples show that $\TC(G)$ can take any value between $\cld(G)$ and $\cld(G\times G)$. Recently the first author with Farber, Lupton and Oprea \cite{FGLO} proved that $\TC(G)\leq \cld_\mathcal{D}(G\times G)$, the Bredon dimension of $G\times G$ with respect to the family $\mathcal{D}$ of subconjugates of the diagonal subgroup $\Delta(G)$. On the one hand, this upper bound in terms of Bredon dimension is only known to improve on the upper bound $\cld(G\times G)$ in the case of free abelian groups (where the value of $\TC(G)$ was already known). On the other hand, no examples of strict inequality $\TC(G)<\cld_\mathcal{D}(G\times G)$ were known. Proposition~\ref{prop:Klein} furnishes such an example, since for a group $G$ acting on itself by conjugation the pair $(G\rtimes G, G)$ is isomorphic to the pair $(G\times G,\Delta(G))$. Therefore for $G$ the fundamental group of the Klein bottle we have $\cld_\mathcal{D}(G\times G)=\cldG(G\rtimes G)=\infty$, while Cohen and Vandembroucq showed in \cite{CV} that $\TC(G)=4$.

\subsection*{Acknowledgements}
We thank the organisers of the British Topology Meeting 2022 and the BIRS workshop 22w5182 during which parts of this work were discussed. We thank Dave Benson for fruitful discussions. 
K.L.\ was supported by the SFB~1085 \emph{Higher Invariants} (Universit\"at Regensburg, funded by the DFG).

\section{Relative group cohomology}\label{sec:defns}

We begin by defining the various flavours of relative group cohomology of a group pair $(K,L)$, and proving some of their basic properties. We work over the integers throughout, so \emph{$K$-module} is synonymous with \emph{$\Z K$-module}. Important examples of $K$-modules are the trivial module $\Z$, the permutation module $\Z(K/L)$, and the module $I(K/L)=\ker (\epsilon: \Z(K/L)\to\Z)$ which arises as the kernel of the obvious augmentation homomorphism. 

\subsection{Takasu cohomology}

\begin{defn}\label{def:Takasu}
Let $M$ be a $K$-module. The \emph{Takasu relative cohomology} of the pair $(K,L)$ with coefficients in $M$ is
\[
H^i(K,L;M):=\Ext^{i-1}_K(I(K/L),M)
\]
for $i\ge1$, and $H^0(K,L;M):=0$.
\end{defn}

The above notion has been studied by Auslander \cite{Aus}, Takasu \cite{Takasu2} and Ribes \cite{Ribes}, and in a more general form with $L$ replaced by a collection of subgroups of $K$ by Wall \cite{Wall}, Bieri--Eckmann \cite{BE} and Alonso \cite{Alonso}, among others. It is clear from the definition that $H^i(K,L;M)$ vanishes when $L=K$, and by a simple dimension-shifting argument it reduces to the ordinary group cohomology $H^i(K;M)$ when $L=\{1\}$ and $i\geq 2$. Another important property is the long exact sequence
\begin{equation}\label{eq:Takasuexact}
\resizebox{\textwidth}{!}{
\xymatrix{
\cdots \ar[r] & H^i(K,L;M) \ar[r] & H^i(K;M) \ar[r]^-{\operatorname{res}} & H^i(L;M|_L) \ar[r] & H^{i+1}(K,L;M) \ar[r] & \cdots
}}
\end{equation}
which holds for any coefficient module $M$. The existence of such a sequence can be deduced from the topological interpretation due to Takasu \cite{Takasu2}:
\[
H^i(K,L;M)\cong H^i(BK,BL;\mathcal{M}),
\]
where on the right-hand side we have the relative singular cohomology of a model for the classifying space $BK$ containing a model for $BL$ as a subspace, with coefficients in the local system $\mathcal{M}$ of abelian groups on $BK$ determined by $M$.  

\begin{defn}\label{def:Takasu_dim}
The \emph{Takasu relative cohomological dimension}, or simply the \emph{Takasu dimension}, of the pair $(K,L)$ is defined by
\[
\cld(K,L):=\sup\{n \mid H^n(K,L;M)\neq 0\mbox{ for some $K$-module }M\}.
\]
\end{defn}

\begin{lem}\label{lem:Tcd1}
The pair $(K,L)$ has $\cld(K,L)\le1$ if and only if $I(K/L)$ is a projective $K$-module.
\end{lem}

\begin{proof} 
By standard arguments $\cld(K,L)=\operatorname{pd} I(K/L)+1$, where $\operatorname{pd}$ denotes the projective dimension of a $K$-module. 
\end{proof}

\subsection{Adamson cohomology}\label{ss:Adamson}

Next we turn to Adamson cohomology. Adamson's original definition \cite{Adamson} was in terms of a natural generalisation of the standard resolution built on the coset space $K/L$. Here we use the interpretation in terms of relative homological algebra due to Hochschild \cite{Hoch}, for which we need to recall some definitions. An epimorphism of $K$-modules $\phi:M\twoheadrightarrow N$ is called a \emph{$(K,L)$-epimorphism} if it splits when regarded as a map of $L$-modules; that is, if there exists an $L$-module map $\psi:N\to M$ such that $\phi\circ \psi =\operatorname{Id}_N$. A $K$-module $P$ is called \emph{$(K,L)$-projective} if the dotted lift exists in any diagram of $K$-modules as below, in which the horizontal arrow is a $(K,L)$-epimorphism:
\[
\xymatrix{
 & P \ar[d] \ar@{..>}[ld] \\
M \ar@{>>}[r] &  N
}
\]

\begin{ex}\label{ex:aug}
The augmentation $\epsilon:\Z(K/L)\twoheadrightarrow \Z$ is a $(K,L)$-epimorphism. Any induced module $\Z K\otimes_L A$ is $(K,L)$-projective (since induction is left adjoint to restriction). In particular, $\Z(K/L)\cong \Z K\otimes_L\Z$ is $(K,L)$-projective.
\end{ex}

An exact sequence consisting of $K$-module maps $\del_i:M_i\to M_{i-1}$ is called \emph{$(K,L)$-exact} if each short exact sequence
\[
\xymatrix{
0 \ar[r] & \operatorname{ker} \del_i \ar[r] & M_i \ar[r] & \operatorname{im} \del_i \ar[r] & 0
}
\]
splits when regarded as an exact sequence of $L$-modules. A \emph{$(K,L)$-projective resolution} of a $K$-module $M$ is then a $(K,L)$-exact sequence
\[
\xymatrix{
\cdots \ar[r] & P_2 \ar[r]^-{\del_2} & P_1 \ar[r]^-{\del_1} &  P_0 \ar[r]^-{\del_0} & M \ar[r] & 0
 }
\]
in which each $P_i$ is $(K,L)$-projective. With these definitions, one can construct the relative Ext functors $\Ext_{(K,L)}^i(-,-)$ in the usual way using $(K,L)$-projective resolutions, and show that they are independent of the chosen resolution \cite[\S 2]{Hoch}.

\begin{defn}\label{def:Adamson}
Let $M$ be a $K$-module. The \emph{Adamson relative cohomology} of the pair $(K,L)$ with coefficients in $M$ is 
\[
H^i([K:L];M):=\Ext^{i}_{(K,L)}(\Z,M).
\]  
\end{defn}

\begin{defn}\label{def:Adamson_dim}
The \emph{Adamson relative cohomological dimension}, or simply the \emph{Adamson dimension}, of the pair $(K,L)$ is defined by
\[
\cld([K:L]):=\sup\{n \mid H^n([K:L];M)\neq 0\mbox{ for some $K$-module }M\}.
\]
\end{defn}

\begin{lem}\label{lem:Acd1}
The pair $(K,L)$ has $\cld([K:L])\le1$ if and only if $I(K/L)$ is a $(K,L)$-projective module.
\end{lem}

\begin{proof} 
Since a short $(K,L)$-exact sequence of $K$-modules induces a long exact sequence of relative Ext groups \cite{Hoch}, Example \ref{ex:aug} and the augmentation sequence
\[
\xymatrix{
0 \ar[r] & I(K/L) \ar[r] & \Z(K/L) \ar[r]^-{\epsilon} & \Z \ar[r] & 0
}
\]
imply that for $i \geq 2$,
\[
H^i([K:L];M)\cong \Ext^{i-1}_{(K,L)}(I(K/L),M).
\]
\end{proof}

One of the most important features of Adamson cohomology is its relation to the ordinary group cohomology of the quotient when $L$ is normal in $K$. In that case, for any $K$-module $M$ we have an isomorphism
\[
H^i([K:L];M)\cong H^i(K/L; M^L),
\]
where $M^L\subseteq M$ denotes the submodule of $L$-invariants~\cite{Adamson}.
Hence $\cld([K:L])= \cld(K/L)$ when $L$ is normal in $K$. On the other hand, in the general case we do not have a long exact sequence such as (\ref{eq:Takasuexact}) in Adamson cohomology.

We can also take Adamson cohomology relative to families of subgroups, by generalising the above definitions in a straightforward manner. By a \emph{family} of subgroups of $K$ we mean a non-empty collection $\F$ of subgroups of $K$ which is closed under conjugation and taking subgroups. Important examples are $\F=\{1\},\FIN,\VCYC$, the families of trivial, finite or virtually cyclic subgroups, respectively, and these have been studied extensively in the literature. The most relevant example for us is the family generated by $L\leq K$: $$\L=\langle L \rangle := \{ H \leq K \mid H\leq kLk^{-1}\mbox{ for some }k\in K\},$$
consisting of all $K$-subconjugates of $L$. Given a family $\F$ of subgroups in $K$, we declare an epimorphism $M\twoheadrightarrow N$ of $K$-modules to be a $(K,\F)$-epimorphism if it splits as a map of $H$-modules, for every $H\in \F$. The definitions of $(K,\F)$-exact and $(K,\F)$-projective follow similarly.

\begin{defn}\label{def:AdamsonF}
Let $M$ be a $K$-module, and $\F$ a family of subgroups of $K$. The \emph{Adamson relative cohomology} of the pair $(K,\F)$, also known as the $\F$-relative cohomology of $K$, with coefficients in $M$ is 
\[
\F H^i(K;M):=\Ext^i_{(K,\F)}(\Z,M).
\]
\end{defn}
Here we have switched notation from the expected $H^i([K:\F];M)$ to be consistent with our references \cite{Nuc,PY}.

\begin{defn}\label{def:AdamsonF_dim}
The \emph{$\F$-relative cohomological dimension} of $K$ is defined by
\[
\F\cld(K):=\sup\{n \mid \F H^n(K;M)\neq 0\mbox{ for some $K$-module }M\}.
\]
\end{defn}

When $L\le K$ and $\L$ is the family generated by $L$, one checks easily that $\L H^i(K;M)\cong H^i([K:L];M)$ and $\L\cld(K)=\cld([K:L])$.

\subsection{Bredon cohomology}

Next we recall the basics of Bredon cohomology of a group with respect to a family of subgroups, for which a nice introductory reference is \cite{Fluch} and a more advanced reference is \cite{Lueck}. Given a family $\F$ of subgroups of $K$, the \emph{$\F$-restricted orbit category of $K$} is defined to be the category $\OF(K)$ whose objects are the transitive $K$-sets $K/H$ for $H\in \F$, and whose morphisms $K/H\to K/H'$ are $K$-equivariant maps. An \emph{$\OF(K)$-module} is a contravariant functor $\ul{M}:\OF(K)\to \mathsf{Ab}$ with values in abelian groups. For example, the constant module $\ul{\Z}$, which takes the value $\Z$ on every object and the identity $\operatorname{Id}:\Z\to\Z$ on every morphism, plays a special r\^ole. The category of $\OF(K)$-modules is an abelian category with enough projectives, with the notions of monomorphism and epimorphism, kernel and cokernel, etc.\ defined object-wise. Therefore we can define the Ext functors $\Ext^i_{\OF(K)}(-,-)$ in the standard way, and show that they are independent of the chosen resolution. 

\begin{defn}\label{def:Bredon}
Let $\ul{M}$ be an $\OF(K)$-module. The \emph{Bredon cohomology of $K$ with respect to the family $\F$} and with coefficients in $\ul{M}$ is defined to be
\[
H^i_\F(K;\ul{M}):=\Ext_{\OF(K)}^i(\ul{\Z},\ul{M}).
\]
In the special case where $\F=\L$ for a subgroup $L\leq K$, we call this the \emph{Bredon relative cohomology} of the pair $(K,L)$.
\end{defn}

\begin{defn}\label{def:Bredon_dim}
The \emph{Bredon cohomological dimension of $K$ with respect to the family $\F$} is
\[
\cld_\F(K):=\sup\{n \mid H^n_\F(K;\ul{M})\neq 0\mbox{ for some $\OF(K)$-module }\ul{M}\}.
\]
In the special case where $\F=\L$ for a subgroup $L\leq K$, we call $\cld_\L(K)$ the \emph{Bredon relative cohomological dimension}, or simply the \emph{Bredon dimension}, of the pair $(K,L)$.
\end{defn} 

Let $\Delta$ be a $K$-set. Define an $\OF(K)$-module $\ul{\Z}\Delta$ by setting $\ul{\Z}\Delta(K/H)=\Z(\Delta^H)$ for each $H\in \F$. There is a notion of freeness for $\OF(K)$-modules, and $\ul{\Z}\Delta$ is free if and only if all isotropy groups of $\Delta$ belong to the family $\F$ (see for example \cite[Ch.\ 1, \S 5-6]{Fluch}). An $\OF(K)$-module is projective if and only if it is a direct summand of a free module; in particular, free $\OF(K)$-modules are projective. 

The above observations permit a topological approach to constructing projective resolutions of $\ul{\Z}$. A $K$-CW complex $X$ is called a \emph{classifying space for $K$ with respect to the family $\F$} if for any $K$-CW complex $Y$ with isotropy in $\F$ there is a unique $K$-map $Y\to X$ up to $K$-homotopy. This is equivalent to the following conditions on the fixed point sets: for all $H\leq K$,
\begin{enumerate}[(i)]
\item $X^H =\emptyset$ if $H\notin\F$,
\item $X^H$ is contractible if $H\in \F$. 
\end{enumerate}
We denote such a classifying space by $E_\F(K)$, while noting that it is only unique up to $K$-homotopy type. Let $\Delta_i$ be the set of $i$-cells of $E_\F(K)$. Then $\ul{C}_i(E_\F(K)):=\ul{\Z}\Delta_i$ is the module of $i$-chains of $E_\F(K)$, and is free as an $\OF(K)$-module by condition (i). The augmented cellular chain complex
\[
\xymatrix{
\cdots \ar[r] & \ul{C}_2(E_\F(K)) \ar[r]^-{\del_2} & \ul{C}_1(E_\F(K)) \ar[r]^-{\del_1} & \ul{C}_0(E_\F(K)) \ar[r]^-{\del_0} & \ul{\Z} \ar[r] & 0
}
\]
is then a free, and hence projective, resolution of $\ul{\Z}$, by condition (ii).  Letting $\gld_\F(K)$ denote the minimal dimension of a $K$-CW complex model of $E_\F(K)$, it follows that $\cld_\F(K)\leq \gld_\F(K)$. L\"uck and Meintrup \cite{Lueck-Meintrup} have shown that $\gld_\F(K)\leq \max\{3,\cld_\F(K)\}$. In particular, we have
\[
\cld_\L(K)\leq \gld_\L(K)\leq\max\{3,\cld_\L(K)\}.
\]

To explain the relationship with Adamson cohomology, we recall that there is an adjoint pair of functors
\begin{equation}\label{eq:adjunct}
\xymatrix{
\OP{res}_{\{1\}}^\F: \OF(K){\text -}\textsf{mod} \ar@<0.5ex>[r] & K{\text -}\textsf{mod}:\OP{coind}_{\{1\}}^\F \ar@<0.5ex>[l]
}
\end{equation}
called \emph{restriction} and \emph{coinduction}. Explicitly,
\[
\OP{res}_{\{1\}}^\F(\ul{M})=\ul{M}(K/1),\qquad \OP{coind}_{\{1\}}^\F(M): K/H\mapsto M^H.
\]
The coinduction functor is also known as the \emph{fixed-points functor}. It is shown in  \cite{Nuc,PY,A-NC-MSS} that applying restriction to an $\OF(K)$-projective resolution
\[
\xymatrix{
\cdots \ar[r] & \ul{P}_2 \ar[r] & \ul{P}_1 \ar[r] & \ul{P}_0 \ar[r] & \ul{\Z} \ar[r] & 0
}
\]
yields a $(K,\F)$-projective resolution
\[
\xymatrix{
\cdots \ar[r] & \OP{res}_{\{1\}}^\F(\ul{P}_2) \ar[r] & \OP{res}_{\{1\}}^\F(\ul{P}_1) \ar[r] & \OP{res}_{\{1\}}^\F(\ul{P}_0) \ar[r] & \Z \ar[r] & 0.
}
\]
Combined with the adjunction (\ref{eq:adjunct}) this implies the following result (see also \cite[Theorem 3.10]{BC-VEB}).

\begin{prop}\label{prop:AleB}
Let $M$ be a $K$-module, and let $\F$ be a family of subgroups of~$K$. Then 
\[
\F H^i(K;M)\cong H^i_\F(K;\OP{coind}_{\{1\}}^\F(M)).
\]
In particular, $\F\cld(K)\leq \cld_\F(K)$ always. When $\L=\langle L\rangle$ is the family generated by a subgroup $L\le K$, we have $\cld([K:L])\leq\cld_\L(K)$.
\end{prop}

For later use we record here Shapiro's Lemma for Bredon cohomology (see \cite[Chapter 3]{Fluch}) and for $\F$-relative cohomology. Let $\F$ be a family of subgroups of $K$, and let $P\leq K$ be a subgroup. Then $\F\cap P$ is a family of subgroups of $P$, and there is an adjoint pair of functors
\[
\xymatrix{
\OP{res}_{\F\cap P}^\F: \OF(K){\text -}\textsf{mod} \ar@<0.5ex>[r] & \mathcal{O}_{\F\cap P}(P){\text -}\textsf{mod}:\OP{coind}_{\F\cap P}^\F \ar@<0.5ex>[l]
}
\]
The restriction functor preserves projectives, implying the following result.

\begin{lem}[Shapiro's Lemma for Bredon cohomology]\label{lem:ShapiroBredon}
Let $\F$ be a family of subgroups of $K$, and let $P\leq K$ be a subgroup. Then for any $\mathcal{O}_{\F\cap P}(P)$-module $\ul{M}$ there are isomorphisms
\[
H^i_{\F\cap P}(P;\ul{M})\cong H^i_\F(K;\OP{coind}^\F_{\F\cap P}(\ul{M})),
\]
natural in $\ul{M}$. In particular, $\cld_{\F\cap P}(P)\leq \cld_\F(K)$.
\end{lem}

\begin{lem}[Shapiro's Lemma for $\F$-relative cohomology]\label{lem:ShapiroRel}
Let $\F$ be a family of subgroups of $K$, and let $P\leq K$ be a subgroup. Then for any $P$-module $M$ there are isomorphisms
\[
(\F\cap P)H^i(P;M) \cong \F H^i(K;\OP{coind}^K_P(M)),
\]
natural in $M$. In particular, $(\F\cap P)\cld(P)\leq \F \cld(K)$.
\end{lem}

\begin{proof}
 This follows from Proposition \ref{prop:AleB} and Lemma \ref{lem:ShapiroBredon}, on noting that the various ways of coinducing up from a $P$-module to an $\OF(K)$-module coincide.   
\end{proof}

Finally in this section, we note that if $L$ is normal in $K$ then $$\cld_\L(K)=\cld(K/L)=\cld([K:L]),$$ where the first equality was shown in \cite[Lemma 2.6]{GMP} and the second was noted in Subsection \ref{ss:Adamson} above.

\section{Equivariant group cohomology}\label{sec:equivdefns}

Let $\G$ be a discrete group. A \emph{$\G$-group} is a discrete group $G$ equipped with an action of $\G$ by automorphisms (equivalently, a homomorphism $\varphi:\G\to \operatorname{Aut} G$). In this section we discuss the various ways to define the equivariant group cohomology of a $\G$-group, each of which uses a different flavour of relative group cohomology applied to the pair $(\GG,\G)$. Here $\GG=G\rtimes_\varphi\G$ is the semi-direct product associated to the action, and $\G$ is identified with the standard subgroup $\{(1,\gamma) \mid\gamma\in\G\}$. 

\begin{defn}[\cite{CG-CO}]\label{def:CG-CO}
Let $M$ be a $\GG$-module. The \emph{Cegarra--Calcines--Ortega equivariant cohomology} of the $\G$-group $G$ with coefficients in $M$ is the Takasu cohomology $H^i(\GG,\G;M)$, see Definition \ref{def:Takasu}.

The \emph{Cegarra--Calcines--Ortega equivariant cohomological dimension}, or just \emph{Cegarra--Calcines--Ortega dimension}, of the $\G$-group $G$ is $\cld(\GG,\G)$.
\end{defn}

This is one of several equivalent definitions of equivariant group cohomology given by Cegarra, Garc\'ia-Calcines and Ortega in \cite{CG-CO}. They derive several properties of this cohomology theory, including a classification of equivariant extensions of $\G$-groups in terms of $H^2(\GG,\G;-)$. In developing this theory they were motivated by applications to graded categorical groups \cite{CG-CO2}, and did not discuss equivariant cohomological dimension.

\begin{defn}[\cite{In1,In2}]\label{def:Inassaridze}
Let $M$ be a $\GG$-module. The \emph{Inassaridze equivariant cohomology} of the $\G$-group $G$ with coefficients in $M$ is the Adamson cohomology
$H^i([\GG:\G];M)$, see Definition \ref{def:Adamson}.

The \emph{Inassaridze equivariant cohomological dimension}, or just \emph{Inassaridze dimension}, of the $\G$-group $G$ is $\cld([\GG:\G])$.
\end{defn} 

Motivated by applications to equivariant algebraic $K$-theory, Inassaridze defined and studied this variant of equivariant group cohomology in \cite{In1} (and more recently in \cite{In2}). Here $H^2([\GG:\G];-)$ may be used to classify equivariant extensions of $\G$-groups which are $\G$-split. Inassaridze also observes \cite[Corollary 5]{In1} that if $G$ is a free group with a basis that is permuted by $\G$, then $\cld([\GG:\G])\leq 1$. We will see in the next section that the converse does not hold. 

Now we let $\mathcal{G}=\langle \G\rangle$ denote the family of subgroups of $\GG$ generated by $\G$.

\begin{defn}[\cite{GMP}]\label{def:GMP}
Let $\ul{M}$ be an $\OGG$-module. The \emph{Bredon equivariant cohomology} of the $\G$-group $G$ with coefficients in $\ul{M}$ is
$H^i_\FG(\GG;\ul{M})$, see Definition \ref{def:Bredon}.

The \emph{Bredon equivariant cohomological dimension}, or just \emph{Bredon dimension}, of the $\G$-group $G$ is $\cld_\FG(\GG)$.
\end{defn}

\begin{cor}[to Proposition \ref{prop:AleB}]\label{cor:AleB}
For any $\G$-group $G$,
\[
\cld([\GG:\G])\leq \cld_\FG(\GG).
\]
\end{cor}

If $\G$ acts trivially on $G$, then $\G\lhd\GG$ is normal with quotient $G$ and $$\cld([\GG:\G])=\cld_\FG(\GG)=\cld(G).$$

\section{Comparison of Inassaridze and Bredon dimensions}\label{sec:AvsB}

Our goal in this section is to give examples of $\G$-groups for which the inequality between Inassaridze and Bredon dimensions given in Corollary \ref{cor:AleB} is strict, thus affirming a conjecture made by Inassaridze \cite[\S 8]{In2}. 

We first note that it is relatively straightforward to find examples of finite groups $\G$ with families $\F$ such that $\F\cld(\G)<\cld_\F(\G)$. The genesis of these examples is the following simple algebraic lemma.

\begin{lem}\label{lem:genesis}
Let $\G$ be a finite group. An epimorphism $\pi:M\to N$ of $\G$-modules splits as a map of $\G$-modules if and only if it splits as a map of $P$-modules for every $p$-subgroup $P\leq \G$ and all primes $p$.
\end{lem}

\begin{proof}
The only if direction is clear. Let $p_1,\ldots ,p_k$ be the primes dividing $|\G|$, and fix Sylow subgroups $P_1,\ldots ,P_k$. Suppose we are given, for each $i=1,\ldots , k$, a $P_i$-splitting $s_i:N\to M$ of $\pi$. We define a morphism of abelian groups $t_i:N\to M$ by setting 
\[
t_i=\sum_{\gamma\in\G/P_i} \gamma s_i \gamma^{-1},
\]
where the sum runs over a set of coset representatives of $P_i$ in $\G$. Since $s_i$ is a morphism of $P_i$-modules, this is independent of the coset representatives chosen. We claim that $t_i$ is in fact a morphism of $\G$-modules. For, given $\sigma\in \G$, we have
\begin{align*}
t_i \sigma & = \sum_{\gamma\in \G/P_i} \gamma s_i \gamma^{-1}\sigma \\
               & = \sum_{\gamma\in \G/P_i} \sigma(\sigma^{-1}\gamma) s_i (\sigma^{-1}\gamma)^{-1} \\
		  & = \sigma \sum_{\sigma^{-1}\gamma\in \G/P_i}(\sigma^{-1}\gamma) s_i (\sigma^{-1}\gamma)^{-1} \\
              & = \sigma t_i
\end{align*}
where the last equality follows from the independence of coset representatives mentioned above.

Observe that $\pi t_i:N\to N$ is multiplication by $|\G/P_i|$. Since the integers $|\G/P_1|,\ldots , |\G/P_k|$ are relatively prime, there exist integers $a_1,\ldots , a_k$ such that $\sum a_i|\G/P_i|=1$. Then $\sum a_it_i$ is a $\G$-splitting of $\pi$.
\end{proof}

\begin{cor}\label{cor:cd0}
Let $\F$ be a family of proper subgroups of the finite group $\G$ which contains all of the $p$-subgroups. Then $\F\cld(\G)=0$, while $\cld_\F(\G)\ge2$.
\end{cor}

\begin{proof}
By Lemma \ref{lem:genesis} any $(\G,\F)$-epimorphism $M\to N$ is $\G$-split. It follows that every $\G$-module is $(\G,\F)$-projective. Hence $\F\cld(\G)=0$. The second claim is \cite[Theorem 1.6]{GMP}. (We note that $\cld_\F(\G)\ge 1$ is an easy consequence of $\G\notin\F$.)
\end{proof}

Now we leverage this argument to give examples of free $\G$-groups $G$ for which $\cld([\GG:\G])=1$ and $\cld_\FG(\GG)>1$.

\begin{defn}\label{def:Gammafree}
A $\G$-group $G$ is called \emph{$\G$-free} if $G$ is a free group with a basis permuted by $\G$.
\end{defn}

\begin{thm}[\cite{GMP}]\label{thm:ESS}
Let $G$ be a $\G$-group with $\G$ finite. Then $\cld_\FG(\GG)\leq 1$ if and only if $G$ is $\G$-free.
\end{thm}

Whether a free $\G$-group is $\G$-free can be detected using non-abelian cohomology. We briefly review the definitions, referring the reader to \cite[I.\S 5]{Serre} for more details. Given a $\G$-group $G$, a \emph{$1$-cocycle} is a map $c:\G\to G$ satisfying $c(\gamma\delta)=c(\gamma){}^\gamma c(\delta)$ for all $\gamma,\delta\in\G$. Two $1$-cocycles $c_1,c_2:\G\to G$ are said to be \emph{equivalent} if there exists $g\in G$ such that $c_1(\gamma)=g^{-1}c_2(\gamma){}^\gamma g$ for all $\gamma\in\G$. This is indeed an equivalence relation, and the set of equivalence classes is denoted $H^1(\G;G)$ and called the \emph{first non-abelian cohomology of $\G$ with coefficients in $G$}. It is a pointed set, with base point given by the class of the trivial $1$-cocycle, which is constant at the identity element of $G$. Any $1$-cocycle in this class is called \emph{principal}. If every $1$-cocycle is principal, we say $H^1(\G;G)$ is trivial and write $H^1(\G;G)=\{1\}$; otherwise we say $H^1(\G;G)$ is non-trivial and write $H^1(\G;G)\neq\{1\}$. 

Note that a $1$-cocycle $c:\G\to G$ determines a subgroup $\tilde{\G}=\{(c(\gamma),\gamma)\mid \gamma\in \G\}$ of $\GG$ which projects isomorphically to $\G$. It is well known that $H^1(\G;G)$ is in bijection with the set of $G$-conjugacy classes of such subgroups of $\GG$ (see for example \cite[p.43]{Serre}).
  
\begin{lem}\label{lem:nonabH1}
Let $G$ be a free $\G$-group with $\G$ finite. Then $G$ is $\G$-free if and only if the first non-abelian cohomology $H^1(P;G)$ is trivial for every subgroup $P\le \G$.
\end{lem}

\begin{proof}
Suppose $H^1(P;G)\neq\{1\}$ for some $P\leq \G$. A non-principal $1$-cocycle $c: P\to G$ determines a finite subgroup $\tilde{P}\leq \GG$
not in the family $\FG$, see \cite[Proposition 4.1]{GMP}. Then by Shapiro's Lemma for Bredon cohomology and \cite[Theorem 1.6]{GMP},
\[
2\leq \cld_{\FG\cap \tilde{P}}(\tilde{P})\leq \cld_\FG(\GG).
\]

To prove the converse, we have to delve a little deeper into the proof of \cite[Theorem 1.5]{GMP}. Since $G\rtimes\G$ is virtually free, it acts on a tree $T$ with finite stabilisers by a result of Karass--Pietrowski--Solitar \cite{KPS}, Cohen \cite{Cohen} and Scott \cite{Scott}. The subgroup $G\leq \GG$ acts freely on $T$, and the subgroup $\G\leq \GG$ fixes some vertex $v_0$ (as must any finite group acting on a tree). Hence the quotient graph $X:=T/G$ inherits an action of $\G$ which fixes a vertex $x_0=[v_0]$, and $\pi_1(X,x_0)$ is isomorphic as a $\G$-group to $G$. 

It remains to show that $\pi_1(X,x_0)$ has a basis permuted by $\G$. For this we show that $X$ is $\G$-homotopy equivalent to a $\G$-CW complex with one vertex, by showing that $X$ admits a $\G$-equivariant spanning tree. By the main result of \cite{KanoSakamoto} (which works also for infinite graphs) it suffices to show that for each vertex $x\in X$ the subgraph $X^{\G_x}$ fixed by the stabiliser subgroup $\G_x$ of $x$ is connected. Now the proof concludes exactly as in the proof of \cite[Theorem 1.5]{GMP}: the assumption that $H^1(\G_x;G)=\{1\}$ allows us to conclude that the $G$-orbit of vertices of $T$ represented by $x$ contains a fixed point.
\end{proof}

\begin{thm}\label{thm:H1Adamson}
Let $G$ be a free $\G$-group with $\G$ finite. Then $\cld([\GG:\G])\le1$ if and only if $H^1(P;G)$ is trivial for all $p$-subgroups $P\le \G$ and all primes $p$.
\end{thm}

\begin{proof}
First suppose that $H^1(P;G)$ is trivial for all $p$-subgroups $P\le \G$. To show that $\cld([\GG:\G])\le1$, it suffices by Lemma \ref{lem:Acd1} to show that the augmentation module $I(\GG/\G)\cong I(G)$ is $(\GG,\G)$-projective. This is equivalent to showing that every $(\GG,\G)$-epimorphism $M\to I(G)$ admits a $\GG$-module splitting.

Let $P\le \G$ be a $p$-subgroup. We may regard $G$ as a free $P$-group, and as such it is $P$-free by Lemma \ref{lem:nonabH1}. Therefore
\[
\cld([G\rtimes P:P])\leq \cld_\mathcal{P}(G\rtimes P) \leq 1,
\]
where $\mathcal{P}$ denotes the family of subgroups of $G\rtimes P$ generated by $P$. Applying Lemma~\ref{lem:Acd1} again, we see that $I(G)$ regarded as a $G\rtimes P$-module is $(G\rtimes P,P)$-projective. 

Any $(\GG,\G)$-epimorphism $\pi:M\to I(G)$ may be regarded as a $(G\rtimes P,P)$-epimorphism via restriction, and therefore admits a $G\rtimes P$-module splitting. Now we mimic the proof of Lemma \ref{lem:genesis} to show that $\pi$ admits a $\GG$-module splitting. Namely, choose Sylow subgroups $P_1,\ldots, P_k$ (one for each prime dividing $|\G|$) and $G\rtimes P_i$-splittings $s_i:I(G)\to M$ of $\pi$. Then set
\[
t_i = \sum_{\gamma\in (\GG/G\rtimes P)} \gamma s_i \gamma^{-1}.
\]
Then there are integers $a_1,\ldots, a_k$ such that $\sum a_i t_i$ is a $\GG$-module splitting of $\pi$. 

To prove the converse, suppose that there exists a $p$-subgroup $P\leq\G$ with $H^1(P;G)\neq\{1\}$. Then there exists a subgroup $\tilde{P}\leq \GG$ which projects isomorphically to $P$ and is not in the family $\FG$. We may take a minimal such $\tilde{P}$, so that any proper subgroup of $\tilde{P}$ is in $\FG$. Then by Shapiro's Lemma \ref{lem:ShapiroRel},
\[
\F\cld(P) = (\FG\cap \tilde{P})\cld(\tilde{P})\leq \cld([\GG:\G]),
\]
where $\F$ is the family of all proper subgroups of the $p$-group $P$. The proof is now completed by the following lemma.
\end{proof}

\begin{lem}\label{lem:infinite}
Let $P$ be a $p$-group and let $\F$ be the family of all proper subgroups of~$P$. Then 
$\F \cld(P)=\infty$.
\end{lem}

\begin{proof}
Assume, on the contrary, that $\F \cld(P)=n<\infty$. Since $P$ is finite, the standard homological algebra shows that there exists a $(P,\F)$-projective resolution of length $n$ by finitely generated $P$-modules
$$
\xymatrix{0 \ar[r] & M_n \ar[r]& M_{n-1} \ar[r] &\cdots \ar[r] &  M_0 \ar[r] & \Z \ar[r] &  0.
}
$$

Write $k=\Z/p\Z$. The $(P,\F)$-projective resolution for $\Z$ splits over $\Z$, and will therefore remain exact when we apply the functor $k\otimes_{\Z} - $. We therefore have an exact sequence
$$
\xymatrix{
0\ar[r] & k\otimes M_{n}\ar[r] & \cdots \ar[r] & k\otimes M_0 \ar[r] &  k\ar[r] & 0,
}
$$
whose Euler characteristic is given by
$$-1+\sum_{i=0}^{n} (-1)^{i}\text{dim}_{k}(k\otimes M_i).$$ Since the sequence is exact, its Euler characteristic is zero. We will show now that $\text{dim}_{k}(k\otimes M_i)$ must be divisible by $p$ for every $i$. By reducing the above equation mod $p$ we will get $-1=0$ mod $ p$, which is a contradiction.

To prove the result about the dimensions we use the characterisation of $(P,\F)$-projective modules \cite[Proposition 2.5]{PY}. Namely, they are all direct summands of modules of the form $M\otimes_{\Z} \Z\Delta$, where $M$ is a $P$-module, and $\Delta$ is a $P$-set with stabilisers in $\F$. This implies that every $k\otimes M_i$ is a direct summand of a module of the form $M\otimes k \Delta$, where $M$ is a $kP$-module, and $\Delta$ is a $P$-set with stabilisers in $\F$. 
It will be enough to prove that any indecomposable direct summand $N$ of $M\otimes k \Delta$ satisfies $p\mid\text{dim}_{k}(N)$. 

For this, we note first that decomposing $\Delta$ to a disjoint union of $P$-orbits enables us to write $M\otimes k\Delta\cong \bigoplus_i M\otimes k P/H_i$ with  $H_i\in \F$. By Krull--Schmidt, $N$ is a direct summand of $M\otimes kP/H_i$ for some $i$. Write $H=H_i$. 
Since $H$ is a proper subgroup of the $p$-group $P$, it is contained in a proper normal subgroup $K\subseteq P$. Take $g\notin K$. We will show that every indecomposable direct summand of $M\otimes k P/H$ as a $k\langle g\rangle$-module has dimension divisible by $p$. By considering the restriction of $N$ to $\langle g \rangle$ and using Krull--Schmidt again, this will prove the result.

Write $p^a$ for the order of $g$ in $P$. Then $k\langle g \rangle\cong k[x]/(x^{p^a})$, by taking $x=g-1$ and using the fact that the characteristic of $k$ is $p$. 
We thus see that every indecomposable $k\langle g \rangle$-module is of the form $k[x]/x^i$ for some $i\in\{1,\ldots, p^a\}$. Since $g\notin K$, and $K$ is a normal subgroup of $P$ that contains $H$, it follows that all the stabilisers of $\langle g\rangle$ in $P/H$ are not $\langle g \rangle$. This implies that $M\otimes kP/H$ will decompose into the direct sum of modules of the form $k[x]/x^i\otimes k\langle g\rangle/ \langle g^{p^b}\rangle$ for some $0<b\leq a$. This is the same as the module $$\text{Ind}^{\langle g \rangle}_{\langle g^{p^b}\rangle}\text{Res}^{\langle g \rangle}_{\langle g^{p^b}\rangle} k[x]/(x^i).$$
The group $\langle g^{p^b}\rangle$ is cyclic as well, and we can write $\text{Res}^{\langle g \rangle}_{\langle g^{p^b}\rangle} k[x]/(x^i)$ as the direct sum of modules of the form $k[y]/y^j$, where $y=g^{p^b}-1$ and $0< j \leq p^{a-b}$. We then have 
$$\text{Ind}^{\langle g \rangle}_{\langle g^{p^b}\rangle}k[y]/y^j\cong k[x]/x^{p^bj},$$ which is indecomposable of dimension $p^bj$, which is divisible by $p$ since $b>0$. We are done.   
\end{proof}

\begin{rem}
The above reveals an interesting dichotomy: the Inassaridze dimension $\cld([\GG:\G])$ when $G$ is a non-trivial free $\G$-group and $\G$ is finite is either $1$ or $\infty$.    
\end{rem}

According to the above, to produce examples with $\cld([\GG:\G])\le 1<\cld_\FG(\GG)$, it is enough to produce examples of free $\G$-groups $G$ for which $H^1(\G;G)\neq\{1\}$, while $H^1(P;G)=\{1\}$ for all $p$-subgroups $P\le\G$. For this we use a theorem of Leary and Nucinkis \cite{Leary-Nucinkis}. Recall that a simplicial complex $L$ is called \emph{flag} if every clique in the $1$-skeleton $L^{(1)}$ spans a simplex. The \emph{right-angled Artin group} $G_L$ associated to $L$ has a presentation with generators corresponding to the vertices of $L$, subject only to the relation that vertices adjacent in $L^{(1)}$ commute with each other. There is a homomorphism $G_L\to \Z$ which sends each generator to $1\in\Z$; the kernel of this homomorphism is the \emph{Bestvina--Brady group} $H_L$ associated to $L$. These constructions are natural with respect to simplicial maps of flag complexes. In particular, if $L$ admits a simplicial action of $\G$, then both $G_L$ and $H_L$ become $\G$-groups in a natural way. The action of $\G$ on $L$ is called \emph{admissible} if the setwise and pointwise stabilisers of each simplex $\sigma\subseteq L$ are equal, and is called \emph{effective} if only the neutral element of $\G$ fixes all of $L$.

\begin{thm}[\cite{Leary-Nucinkis}]\label{thm:LN}
Let $L$ be a non-empty finite flag complex with an effective admissible action of a finite group $\G$. Then for every subgroup $P\le\G$,
we have $H^1(P;H_L)=\{1\}$ if and only if $L^P\neq \emptyset$.
\end{thm}

Although the statement in \cite{Leary-Nucinkis} concerns conjugacy classes of finite subgroups of $H_L\rtimes\G$, it is easily seen to imply the above statement in terms of non-abelian cohomology. Note that there is no requirement that $L$ be connected. In particular, when $L$ is a finite $\G$-set, considered as a simplicial complex of dimension $0$, $G_L$ is a finitely generated $\G$-free $\G$-group, and $H_L$ is a $\G$-subgroup which need not be (finitely generated or) $\G$-free.

\begin{ex}\label{ex:C6}
Let $L=\{a,b,c,d,e\}$ with the cyclic group $C_6$ acting by
\[
a\mapsto b\mapsto c\mapsto a,\qquad d\mapsto e \mapsto d.
\]
 Theorem \ref{thm:LN} then implies that $H^1(C_6;H_L)\neq\{1\}$ while $H^1(P;H_L)=\{1\}$ for all $p$-subgroups $P\leq C_6$. Then Theorem \ref{thm:H1Adamson} implies that $\cld([H_L\rtimes C_6:C_6])=1$, while Theorem \ref{thm:ESS} and Lemma \ref{lem:nonabH1} imply that $\cld_\FG(H_L\rtimes C_6)>1$.

 In fact it can be seen that in this example the Bredon dimension is infinite. By construction there is a finite subgroup $\tilde{C_6}\leq H_L\rtimes C_6$ that projects isomorphically to $C_6$, is not in $\FG$, but all of whose subgroups are in $\FG$. Hence by Shapiro's Lemma \ref{lem:ShapiroBredon},
\[
\cld_\F(C_6)=\cld_{\FG\cap\tilde{C_6}}(\tilde{C_6})\leq \cld_\FG(H_L\rtimes C_6),
\]
where $\mathcal{F}$ denotes the family of all proper subgroups of $C_6$. Next we use \cite[Lemma 5.2]{GMP}, which states that if $N\lhd\G$ is the kernel of $p:\G\to \G/N$ and $\F$ is any family of subgroups of $\G$, then
\[
\cld_{\F_N}(\G/N)\leq \cld_\F(\G),
\]
where $\F_N=\{L\leq \G/N \mid p^{-1}(L)\in \F\}$. Applying this with $N=C_3$ and $\G=C_6$ with $\F$ the family of proper subgroups gives
\[
\infty=\cld_{\{1\}}(C_2)\leq \cld_\mathcal{F}(C_6)\leq \cld_\FG(H_L\rtimes C_6).
\]
\end{ex}

\begin{ex}\label{ex:A5}
Let $M$ be the Floyd--Richardson example \cite{FloydRichardson} of an acyclic $2$-dimensional complex with an admissible $A_5$-action without global fixed points (an example which is ubiquitous in the geometric group theory literature). Taking $L$ to be the $0$-skeleton of $M$, we again find that $L^{A_5}=\emptyset$ while $L^P\neq \emptyset$ for all $p$-subgroups $P\leq A_5$. It follows that $\cld([H_L\rtimes A_5:A_5])=1$ while $\cld_\FG(H_L\rtimes A_5)>1$. 

We will in fact now show that $\cld_\FG(H_L\rtimes A_5) \leq 3$. Consider the families~$\FG\subset \FIN$ of subgroups of~$H_L\rtimes A_5$ and apply the Bredon cohomology spectral sequence argument of ~\cite[Corollary~4.1]{M-P1}.
This gives
\[
    \cld_\FG(H_L\rtimes A_5)\leq \cld_\FIN(H_L\rtimes A_5) + \sup_{S\in \FIN} \cld_{\FG\cap S}(S).
\]
Let~$S\leq H_L\rtimes A_5$ be any finite subgroup.
Since $L^{A_5}=\emptyset$ and $L^W\neq \emptyset$ for every proper subgroup~$W\leq A_5$, there are the following two cases (see~\cite[Theorem~3]{Leary-Nucinkis}).
Either~$S$ is subconjugate to~$A_5$, in which case $\cld_{\FG\cap S}(S)=0$.
Or~$S$ maps isomorphically to~$A_5$ but is not conjugate to~$A_5$, while all proper subgroups of~$S$ are subconjugate to~$A_5$. 
In this case $\cld_{\FG\cap S}(S)=\cld_\mathcal{P}(A_5)$, where~$\mathcal{P}$ is the family of all proper subgroups of~$A_5$.
By \cite[Example 5.1]{Adem} and \cite[Theorem 1.6]{GMP}, one has $\cld_\mathcal{P}(A_5) = 2$.
Since~$H_L\rtimes A_5$ is virtually free, $\cld_\FIN(H_L\rtimes A_5)=1$.
Together, we conclude that $\cld_\FG(H_L\rtimes A_5) \leq 1+2=3$. On the other hand, we already know that $\cld_\FG(H_L\rtimes A_5) \geq 2$. Hence, $\cld_\FG(H_L\rtimes A_5)$ is either $2$ or $3$. We do not know the exact value.
\end{ex}

Mart\'{i}nez-P\'{e}rez \cite{M-P} and Degrijse--Petrosyan \cite{DP} have given examples of virtually torsion-free groups whose virtual cohomological dimension is less than their Bredon dimension with respect to the family of finite subgroups. We now describe how their examples fit our framework, giving examples of $\G$-groups with $1<\cld([\GG:\G])<\cld_\FG(\GG)$. 

Let $\FIN$ denote the family of finite subgroups. One of the main results of \cite{M-PN} states that for a virtually torsion-free group, the relative dimension $\FIN\cld$ and the 
virtual cohomological dimension
coincide.

\begin{lem}[{\cite[Proposition 4.1]{GMP}}]\label{lem:GequalsFIN}
Let $G$ be a $\G$-group, where $G$ is torsion-free and $\G$ is finite. Then $\FG=\FIN$ if and only if $H^1(P;G)=\{1\}$ for every subgroup $P\leq\G$.
\end{lem}

\begin{ex}\label{ex:M-P}
Let $p$ be an odd prime and let $C_p$ denote the cyclic group of order~$p$. According to \cite[Example 3.6]{M-P}, there is a finite flag complex $L$ equipped with an admissible effective $C_p$-action having global fixed points (i.e., $L^{C_p}\neq\emptyset$) such that the equivariant Bestvina--Brady group $H_L$ is torsion-free and satisfies
\[
\FIN\cld(H_L\rtimes C_p) = 3 < 4 = \cld_{\FIN}(H_L\rtimes C_p).
\]
By Theorem \ref{thm:LN} we have $H^1(P;H_L)=\{1\}$ for every subgroup $P\leq C_p$, and therefore by Lemma \ref{lem:GequalsFIN} the families $\FIN$ and $\FG$ of the semi-direct product $H_L\rtimes C_p$ coincide, giving
\[
\FG\cld(H_L\rtimes C_p)=\cld([H_L\rtimes C_p:C_p])=3<4=\cld_\FG(H_L\rtimes C_p).
\]
\end{ex}

\begin{ex}\label{ex:3n4n}
Let $p_1,\ldots, p_n$ be distinct odd primes. For $i=1,\ldots, n$, let $L_i$ be a finite flag complex as in Example \ref{ex:M-P}. Set $G:= \prod_{i=1}^n H_{L_i}$ and $\Gamma:=\prod_{i=1}^n C_{p_i}$, and regard $G$ as a $\G$-group with the product action. Then according to \cite[Theorem~C]{DP}, we have
\[
\FIN\cld(\GG)=3n<4n=\cld_\FIN(\GG).
\]
To apply Lemma \ref{lem:GequalsFIN} and conclude that $\FIN=\FG$ in this example, we need to verify that $H^1(P;G)=\{1\}$ for every subgroup $P\le\G=\prod_{i=1}^n C_{p_i}$. Since the primes $p_i$ are distinct, an arbitrary subgroup $P$ of $\G$ is of the form $\prod_{j\in J} C_{p_j}$ for some $J\subseteq \{1,\ldots , n\}$. By repeated application of Lemma \ref{lem:H1prod} below, we have
\[
H^1(P;G)\cong \left(\prod_{j\in J} H^1(C_{p_j};H_{L_j})\right) \times\left(\prod_{i\notin J} H^1(\{1\};H_{L_i})\right).
\]
Since $H^1(C_{p_j};H_{L_j})=\{1\}$ by Theorem \ref{thm:LN}, we see that $H^1(P;G)=\{1\}$ for all $P\leq \G$. Finally, we conclude that 
 \[
\FG\cld(\GG)=\cld([\GG:\G])=3n<4n=\cld_\FG(\GG).
\]
\end{ex}

\begin{lem}\label{lem:H1prod}
For $i=1,2$, let $\G_i$ be a finite group and let $G_i$ be a torsion-free $\G_i$-group. Then
\[
H^1(\G_1\times\G_2;G_1\times G_2)\cong H^1(\G_1;G_1)\times H^1(\G_2;G_2).
\]
\end{lem}

\begin{proof}
Recall that $H^1(\G;G)$ is in bijection with the set of subgroups of $\GG$ that project isomorphically to $\G$, up to conjugacy by elements of $G$.

Consider a subgroup $$A\leq (G_1\times G_2)\rtimes(\G_1\times \G_2)\cong (G_1\rtimes\G_1)\times(G_2\rtimes\G_2)$$ that projects isomorphically onto $\G_1\times\G_2$. We claim that if $A$ contains an element of the form $\big((g_1,g_2),(\gamma_1,1)\big)$ then $g_2 = 1$. Indeed, if $\big((g_1,g_2),(\gamma_1,1)\big)\in A$ then it has finite order, because $A$ is finite. 
The image of this element under the projection onto $G_2\rtimes \G_2$ is $(g_2,1)$. This element has finite order as well, but since $G_2$ is torsion-free it must be trivial. 
Write now $A_i$ for the subgroup of $A$ that projects isomorphically onto $\G_i$ for $i=1,2$. It follows that $A_i$ is contained in $G_i\rtimes \G_i$ for $i=1,2$ and that $A$ splits as $A_1\times A_2$. Moreover, two such subgroups $A$ and $B$ are conjugate by an element of $G_1\times G_2$ if and only if $A_1$ is conjugate to $B_1$ by an element of $G_1$ and $A_2$ is conjugate to $B_2$ by an element of $G_2$. This gives the desired bijection.
\end{proof}

\section{Comparison of Cegarra--Calcines--Ortega and Bredon dimensions}\label{sec:TvsB}

In this section we provide examples of $\G$-groups for which the Cegarra--Calcines--Ortega dimension is strictly greater than the Bredon dimension (and therefore also than the Inassaridze dimension), and where the Cegarra--Calcines--Ortega dimension is strictly less than the Inassaridze dimension (and therefore also than the Bredon dimension). The following two examples were already considered in \cite[\S 7]{A-NC-M}.

\begin{ex}\label{ex:(ZZ,Z)}
Consider the pair $(K,L)=(\Z\times\Z,\Z)$ where the subgroup is the second factor. Its relative group cohomology is the equivariant group cohomology of $\Z$ regarded as a $\Z$-group with trivial action. As noted at the end of Section \ref{sec:equivdefns}, this implies that 
$$\cld_\FG(\Z\times\Z)=\cld([\Z\times\Z:\Z])=\cld(\Z)=1.$$
 For the Cegarra--Calcines--Ortega dimension we note that from the long exact sequence (\ref{eq:Takasuexact}) and the fact that $\cld(\Z\times\Z)=2>1=\cld(\Z)$ it is easily deduced that $\cld(\Z\times\Z,\Z)=\cld(\Z\times\Z)=2$. So in this example the Cegarra--Calcines--Ortega dimension exceeds the Bredon dimension. 

More generally, if $L$ is normal in $K$ with $\cld(L)<\cld(K)$ then $\cld_\L(K)=\cld(K/L)$, while $\cld(K,L)=\cld(K)$.
\end{ex}

\begin{ex}\label{ex:(Z,2Z)}
Consider the pair $(K,L)=(\Z,2\Z)$. Notice that this example does not fit into the framework of equivariant group cohomology, since $2\Z$ is not a retract of $\Z$. By normality, both the Adamson and the Bredon dimension equal $\cld(\Z/2\Z)$, hence are infinite. We claim that $\cld(\Z,2\Z)=2$. Since $\cld(\Z)=\cld(2\Z)=1$, the long exact sequence (\ref{eq:Takasuexact}) implies immediately that $\cld(\Z,2\Z)\le2$. Taking coefficients in the sign module $\Z^{-}$ gives an exact sequence
\[
\xymatrix{
\cdots \ar[r] & H^1(\Z;\Z^{-}) \ar[r] \ar@{=}[d] & H^1(2\Z;\Z^{-}|_{2\Z}) \ar[r] \ar@{=}[d] & H^2(\Z,2\Z;\Z^{-})\\
& \Z/2\Z & \Z
}
\]
which shows that $H^2(\Z,2\Z;\Z^{-})\neq 0$, and hence $\cld(\Z,2\Z)=2$.
\end{ex}

Example \ref{ex:(ZZ,Z)} fits into the framework of equivariant group cohomology, but for a $\G$-group with trivial action. We would like to provide examples of $\G$-groups with non-trivial actions for which the Cegarra--Calcines--Ortega dimension exceeds the Bredon dimension, and for which the Inassaridze dimension exceeds the Cegarra--Calcines--Ortega dimension.

In order to achieve our first goal, we prove an equivariant Stallings--Swan type result for Cegarra--Calcines--Ortega equivariant group cohomology. This will be deduced from the following result, whose statement is due to Alonso \cite{Alonso}. An action of a group on a complex is said to be \emph{$0$-unfree} if the stabiliser of any positive dimensional cell is trivial.

\begin{thm}[{\cite[Theorem 3]{Alonso}}]\label{thm:DA}
A group pair $(K,L)$ has $\cld(K,L)\le n$ if and only if $K$ acts admissibly and $0$-unfreely on an acyclic $n$-dimensional complex with a single orbit of vertices of type $K/L$.
\end{thm}

The case $n=1$ is equivalent (via the Bass--Serre theory of groups acting on trees) to a result proved by Dicks \cite[Theorem IV.2.11]{Dicks}. Alonso proves a more general statement involving group pairs $\big(K,(L_i)_{i\in I}\big)$ and non-transitive actions.

\begin{defn}\label{def:stronglygammafree}
A $\G$-group $G$ is called \emph{strongly $\G$-free} if $G$ is free with basis a free $\G$-set.
\end{defn}

\begin{thm}\label{thm:Stallings-Swan-Takasu}
Let $G$ be a $\G$-group. Then $\cld(G\rtimes\G,\G)\le1$ if and only if $G$ is strongly $\G$-free.
\end{thm}

\begin{proof}
Suppose $G$ is free with basis a free $\G$-set $S$. Consider the Cayley graph $T=T(G,S)$, with vertex set $G$ and edges $(g,gs)$ for $g\in G$, $s\in S$. The standard action of $G$ on $T$ is free. There is also an action of $\G$ on $T$, coming from the action of $\G$ on $G$ (this uses that $S$ is a $\G$-set).  These actions are compatible, so together give an action of $\GG$ on $T$. Explicitly, an element $(h,\gamma)\in \GG$ sends an edge $(g,gs)$ to the edge $(h{}^\gamma g,h{}^\gamma g{}^\gamma s)$. This action is vertex-transitive, since this is true already of the $G$-action. The stabiliser of a vertex $g\in G$ is the $(g,1)$-conjugate of the stabiliser of the identity $1\in G$, which is easily seen to equal $\G$. The stabiliser of an edge $(g,gs)$ is the $(g,1)$-conjugate of the stabiliser of $(1,s)$, which is trivial since $\G$ acts freely on $S$. This shows that the action is $0$-unfree, and by Theorem~\ref{thm:DA} we have that $\cld(\GG,\G)\le1$.

Conversely, suppose that $\cld(\GG,\G)\le1$. By Theorem \ref{thm:DA} the group $\GG$ acts admissibly and $0$-unfreely on a tree $T$ with a single orbit of vertices of type $\GG/\G\cong G$. Since no non-trivial element of $G$ is conjugate in $\GG$ to an element of $\G$, it follows that $G$ acts freely on $T$. The quotient graph $X:=T/G$ has fundamental group $G$, and by the above has a single vertex and is therefore a bouquet of circles. Thus $G$ is freely generated by the oriented edges of $X$. The $\GG$-action on $T$ descends to a $\G$-action on $X$, and it remains only to check that this action freely permutes the oriented edges.

An edge of $X$ is the $G$-orbit $[e]$ of an edge $e$ of $T$. Suppose $\gamma\in\G$ stabilises $[e]$. Then there is an element $g\in G$ such that ${}^\gamma e=ge$. Therefore $(g^{-1},\gamma)\in \GG$ stabilises $e$, and by the $0$-unfreeness assumption $(g^{-1},\gamma)$, and therefore $\gamma$ must be the identity. This proves the claim and the theorem.      
\end{proof}

Combining this result with Theorem \ref{thm:ESS}, we can construct many examples of $\G$-groups $G$ with 
\[
\cld(\GG,\G)>\cld_\FG(\GG)=1.
\]
In fact, this will be true for any $\G$-group $G$ which is $\G$-free but not strongly $\G$-free. While this is obviously true for free groups with trivial action, it is also true whenever the rank of $G$ is less than $|\G|$. 

\begin{ex}\label{ex:C6Takasu}
 Consider the action of $C_6$ on $L=\{a,b,c,d,e\}$ described in Example~\ref{ex:C6}, for which the associated right-angled Artin group $G_L$ is a free $C_6$-group of rank~$5$. Since $L$ is a $C_6$-basis, $G_L$ is $C_6$-free and has Bredon dimension $\cld_\FG(G_L\rtimes C_6)=1$. Here the Cegarra--Calcines--Ortega dimension $\cld(G_L\rtimes C_6,C_6)$ is infinite, as can be seen by examining the long exact sequence (\ref{eq:Takasuexact}) with coefficients in the trivial module $\Z$, together with the Lyndon--Hochschild--Serre spectral sequence for the extension $1\to G_L\to G_L\rtimes C_6\to C_6\to 1$, which shows that the restriction $H^i(G_L\rtimes C_6;\Z)\to H^i(C_6;\Z)$ has non-trivial kernel in infinitely many degrees.
\end{ex}

 As another application of Theorem \ref{thm:Stallings-Swan-Takasu} we are able to identify the projective objects in the category $\G{\text -}\mathsf{Grp}$ of $\G$-groups. Recall that an object $P$ in a category $\mathcal{C}$ is called \emph{projective} if any morphism $P\to N$ in $\mathcal{C}$ admits a lift through any epimorphism $M\twoheadrightarrow N$ in $\mathcal{C}$. Since $\G{\text -}\mathsf{Grp}$ is a concrete category surjective $\G$-homomorphisms are in particular epimorphisms. On the other hand, it follows from \cite{Lind} that an epimorphism in the category $\mathsf{Grp}$ of groups is surjective. Using the right adjoint functor from $\mathsf{Grp}$ to $\G{\text -}\mathsf{Grp}$ of the forgetful functor, we see that an epimorphism in the category $\G{\text -}\mathsf{Grp}$ is also surjective. Now one sees by standard arguments that strongly $\G$-free $\G$-groups are projective, and that any projective object is a retract of a strongly $\G$-free $\G$-group. Note that it is not neccessarily true that a $\G$-subgroup of a strongly $\G$-free $\G$-group is strongly $\G$-free (or even $\G$-free). However, Theorem \ref{thm:Stallings-Swan-Takasu} and the naturality properties of Takasu cohomology allow us to deduce that any retract of a strongly $\G$-free $\G$-group has Cegarra--Calcines--Ortega dimension less than or equal to $1$, hence is strongly $\G$-free.

\begin{cor}\label{cor:projectives}
The projective objects in $\G{\text -}\mathsf{Grp}$ are the strongly $\G$-free $\G$-groups.
\end{cor}

Finally, we want to give examples of $\G$-groups $G$ with Cegarra--Calcines--Ortega dimension less than their Inassaridze dimension, and hence also than their Bredon dimension.

We consider $\G=G$ acting on itself via the conjugation action $\operatorname{Ad}:G\to \operatorname{Aut}(G)$. There is an isomorphism of pairs 
\[
(G\rtimes_{\operatorname{Ad}}G,G) \stackrel{\cong}{\longrightarrow} (G\times G, \Delta(G)),\qquad (g,h)\mapsto (gh,h),
\]
where $\Delta(G)\leq G\times G$ is the diagonal subgroup. Therefore the equivariant group cohomology of $G$ with the conjugation action coincides with the relative group cohomology of the pair $(G\times G, \Delta(G))$. 

\begin{rem}
Connections between the Adamson and Bredon cohomologies of the pair $(G\times G,\Delta(G))$ and Farber's topological complexity $\TC(G)$ of the group $G$ have been explored in \cite{FGLO,BC-VEB}. In particular, letting $\mathcal{D}:=\langle\Delta(G)\rangle$ denote the family generated by the diagonal subgroup, we have
\begin{equation}\label{eq:TCbounds}
\cld_{\{1\}\hookrightarrow\mathcal{D}}(G\times G)\leq \TC(G)\leq \cld_\mathcal{D}(G\times G).
\end{equation}

(Here $\cld_{\{1\}\hookrightarrow\mathcal{D}}(G\times G)$ denotes the maximum $n$ such that for some $\OD(G\times G)$-module $\ul{M}$ the homomorphism $H^n_\mathcal{D}(G\times G;\ul{M})\to H^n(G\times G; \operatorname{res}^\mathcal{D}_{\{1\}}(\ul{M}))$ induced by the restriction functor 
\[
\operatorname{res}^\mathcal{D}_{\{1\}}: \OD(G\times G){\text -}\textsf{mod}\to (G\times G){\text -}\textsf{mod}
\]
is nonzero.) 
It is an open question posed by Farber in 2006~\cite{Far06} to find an algebraic characterisation of~$\TC(G)$.
In the examples given below $\TC(G)$ is strictly less than $\cld([G\times G:\Delta(G)])$ (and therefore also than $\cld_\mathcal{D}(G\times G)$). This means that $\TC(G)$ does not coincide with either the Adamson or Bredon dimension of the pair~$(G\times G,\Delta(G))$ in general.
On the other hand, we do not know of an example for which the lower bound given in~\eqref{eq:TCbounds} is strict.

Note that when $G$ is free abelian we have 
\[
\TC(G)=\cld(G)<\cld(G\times G)=\cld(G\times G,\Delta(G)),
\]
showing that $\TC(G)$ also does not generally coincide with the Takasu dimension of the pair~$(G\times G,\Delta(G))$.
\end{rem}

\begin{ex}\label{ex:Klein}
Let $G=\langle a,b\mid a^2=b^2\rangle$ be the fundamental group of the Klein bottle. Let $c=a^2=b^2$, and note that $G\cong \langle a\rangle \ast_{\la c\ra}\la b\ra$ decomposes as an amalgam. In particular, if $a^m$ and $b^n$ are conjugate in $G$ for some integers $m,n\in \Z$, then in fact $m=n=2k$ for some $k\in \Z$ and $a^m=b^n=c^k$.

We consider $G$ as a $G$-group, with $G$ acting on itself via conjugation. We claim that 
\[
\cld([G\rtimes G: G])=\cld_\FG(G\rtimes G)=\infty, \qquad \cld(G\rtimes G,G)=4.
\]
As noted above, this is equivalent to showing that 
\[
\cld([G\times G: \Delta(G)])=\cld_\FD(G\times G)=\infty, \qquad \cld(G\times G,\Delta(G))=4.
\]
Since $\cld(G)=2$ and $\cld(G\times G)=4$, the latter equality is evident from the long exact sequence (\ref{eq:Takasuexact}) of the pair $(G\times G,\Delta(G))$ (compare Example \ref{ex:(ZZ,Z)}). The other equalities follow from Shapiro's Lemma applied to the subgroup
\[
\delta = \{(a^m,b^m) \mid m\in \Z\}\leq G\times G
\]
generated by the element $(a,b)$. We give the proof for Bredon cohomology; the proof for Adamson cohomology is similar. 

We claim that $\FD\cap \delta$ is the family generated by $\la (c,c) \ra$. Note that $\la (c,c) \ra$ is clearly in $\FD\cap \delta$, so it suffices to show that an element of $\delta$ contained in a subgroup in the family $\FD$ is contained in $\la (c,c)\ra$. This follows from the structure as an amalgam, since if $a^m$ is conjugate to $b^m$ in $G$ then $m=2k$ and $(a^m,b^m)=(c^k,c^k)$. 

Now note that since $\la (c,c)\ra$ is normal in $\delta$ with quotient $\Z/2\Z$, Shapiro's Lemma gives
\[
\infty=\cld(\Z/2\Z)=\cld_{\FD\cap\delta}(\delta)\leq \cld_\FD(G\times G).
\]

Finally, Cohen and Vandembroucq have shown that $\TC(G)=4$ \cite{CV}.
\end{ex}

\begin{ex}
\label{ex:amalg}
The above example can be generalised as follows. If $G=A\ast_C A$ is an amalgam with $C$ normal in $A$, then 
\begin{equation}\label{eq:amalg}
\cld_\FG(G\rtimes G)=\cld_\FD(G\times G)\geq\cld(A/C).
\end{equation}
In particular, if $C$ has finite index in $A$ then both $\cld_\FG(G\rtimes G)$ and $\cld_\FD(G\times G)$ are infinite. On the other hand, the cohomological dimension of $G$ may well be finite, which forces the Takasu dimensions $\cld(G\rtimes G,G)$ and $\cld(G\times G,\Delta(G))$ and the topological complexity $\TC(G)$ to be finite. 

An interesting application of~\eqref{eq:amalg} with $A/C$ torsion-free is given by the commutator subgroup~$[F_n,F_n]$ of the free group~$F_n$ of rank~$n$.
By~\eqref{eq:amalg} the group~$G_n\coloneqq F_n\ast_{[F_n,F_n]} F_n$ satisfies $\cld_\FD(G\times G)\geq \cld(\mathbb{Z}^n)=n$.
On the other hand, we have $\TC(G_n)\leq \cld(G_n\times G_n)\leq 4$ indepently of~$n$.
Hence for~$n>4$ the group~$G_n$ satisfies $\TC(G_n)<\cld_\FD(G_n\times G_n)$, showing again that the upper bound in~\eqref{eq:TCbounds} can be strict.
We leave the details to an interested reader.
\end{ex}

\end{document}